\documentclass[12pt,a4paper]{article}
\usepackage[width=19cm,height=23cm]{geometry}
\usepackage{amsmath,amssymb,amsfonts,amsthm,enumerate,float}
\usepackage{mathrsfs,amsbsy,bm,mathtools}
\usepackage{indent first}
\usepackage{scalerel}

\newtheorem{Thm}{Theorem}[section]%

\theoremstyle{definition} 
\newtheorem{Def}[Thm]{Definition}%
\newtheorem{Rmk}[Thm]{Remark}%
\newtheorem{Conj}[Thm]{Conjecture}%
\newtheorem{Ques}[Thm]{Question}
\newtheorem{Obs}[Thm]{Observation}

\title{\bf Entire Curves Producing
\\
Distinct  Nevanlinna Currents }
\author{Song-Yan Xie}

\begin{document}
\maketitle

\centering
{\em dedicated to Julien Duval with admiration
}

{\let\thefootnote\relax}\par

\medskip
\abstract{
First, inspired by a question of Sibony, we show that in every compact complex manifold $Y$ with certain  Oka property,  there exists some entire curve $f: \mathbb{C}\rightarrow Y$ generating
all  Nevanlinna/Ahlfors currents on $Y$, by holomorphic discs  $\{f\restriction_{\mathbb{D}(c, r)}\}_{c\in \mathbb{C}, r>0}$.
Next, we answer positively a question of Yau, by constructing 
 some entire curve $g: \mathbb{C}\rightarrow X$
in the  product
$X:=E_1\times E_2$ of two elliptic curves $E_1$ and $E_2$, such that by using concentric
holomorphic discs  $\{g\restriction_{\mathbb{D}_{ r}}\}_{r>0}$  we can obtain infinitely many distinct  Nevanlinna/Ahlfors currents proportional to the extremal currents of integration along curves $[\{e_1\}\times E_2]$,  $[E_1\times \{e_2\}]$  for all $e_1\in E_1, e_2\in E_2$  simultaneously.
This phenomenon is new, and it shows 
tremendous holomorphic flexibility of entire curves in large scale geometry.

}

\section{\bf Introduction}

In 1967, Shoshichi Kobayashi introduced a pseudo-distance $d_X$ intrinsically associated with any complex manifold $X$. When  $d_X$ is a true distance, $X$ is called Kobayashi hyperbolic (cf.~\cite{Kobayashi-1998}). In the compact case, by Brody lemma~\cite{Brody-1978}, $X$ is Kobayashi hyperbolic if and only if it contains no entire curve, i.e., nonconstant holomorphic map from the complex line $\mathbb{C}$ into $X$.
Observing that most  Riemann surfaces are  hyperbolic, Kobayashi expected the same phenomenon for ``most'' higher dimensional complex manifolds~\cite{Kobayashi-1998}.
 In this direction, the leading problem  is the famous Green-Griffiths conjecture~\cite{Green-Griffiths}, which
 stipulates that for every complex projective
variety $X$ of general type, all entire curves $f: \mathbb{C}\rightarrow X$ shall be factored through certain  proper algebraic subvariety $Y\subsetneqq X$.

 Partially advertised by the philosophical analogy
 between the value distribution of entire curves in Nevanlinna theory and the  locations of rational points in  Diophantine geometry
(cf. Vojta's dictionary~\cite{Vojta-1999}),
 the Green-Griffiths conjecture has attracted  much attention in the past decades.
 One key technique, which was introduced by A. Bloch~\cite{Bloch-1926} about a century ago, is by using negatively twisted $k$-jet differentials $\omega_1, \omega_2, \omega_3,\dots$ on $X$ as obstructions for the existence of entire curves $f: \mathbb{C}\rightarrow X$, since $f$ must obey the ordinary differential equations $f^*\omega_1\equiv 0, f^*\omega_2\equiv 0, f^*\omega_3\equiv 0, \dots$. The reason is that the so called  {\sl fundamental vanishing theorem of entire curves} (cf. \cite{Green-Griffiths, Siu-1995, Ru-2021}).   
The existence of such
nontrivial negatively twisted $k$-jet differential forms was  established by Merker~\cite{Merker-2015}   for general type smooth hypersurfaces in projective spaces, 
and then by Demailly~\cite{Demailly-2011} for general type compact complex manifolds, both requiring $k\gg 1$. One natural idea, which goes back at least to~\cite{Green-Griffiths}, is  ``controlling'' the common zero loci defined by $\omega_1, \omega_2,  \omega_3, \dots$ in the $k$-jet space of $X$ for concluding the desired hyperbolicity.
However, in practice, such approach is very difficult. For instance, in the simplest case  $k=1$, there was a related conjecture of Debarre~\cite{Debarre-2005} that 
for general  $c\geqslant   N/2$ hypersurfaces $H_1, \dots, H_c \subset \mathbb{CP}^N$ of large
degrees $\gg 1$, the intersection $X:= H_1\cap \cdots\cap H_c$ shall have
ample cotangent bundle $T_X^*$, in particular, the negatively twisted $1$-jet differentials on $X$ are abundant. The first proof of the Debarre ampleness conjecture was established in~\cite{Xie-2018}, 
using explicit $1$-jet differentials obtained by Brotbek~\cite{Brotbek-2016} (see also another proof~\cite{Brotbek-Lionel} appeared shortly later).  
For Green-Griffiths conjecture on general hypersurfaces of large degrees in projective spaces, Siu introduced 
a strategy~\cite{Siu-ICM2002} of using
slanted vector fields (cf. \cite{Merker-2009}), which explores certain symmetry of $k$-jet spaces of universal hypersurfaces,    to facilitate  controlling base loci of negatively twisted $k$-jet differentials.
The existence  of negatively twisted $k$-jet differentials for suitably small $k$  can be
established by certain  Riemann-Roch calculations, 
 see \cite{DMR, Berczi-Kirwan} and references therein for the developments and state of the art.
  A breakthrough of Riedl and Yang~\cite{Riedl-Yang} discovered  that
  the Green-Griffiths conjecture for general hypersurfaces with smaller degrees can imply the
  Kobayashi hyperbolicity conjecture for general  hypersurfaces with larger degrees (compare with~\cite{Siu-2015, Brotbek-2017}).
  
While  the Green-Griffiths conjecture for  $\geqslant 3$ dimensional varieties remains dubious, it is likely to be true for compact projective surfaces of general type. One strong evidence is the celebrated work~\cite{McQuillan-1998} of McQuillan, in which he verified the case of compact surfaces $X$ with Chern number inequality $c_1^2>c_2$ (see also~\cite{Lu-Yau}).
This extra condition guarantees the existence of one negatively twisted $1$-jet differential $\omega$, thus
inducing a multi-foliation $\mathcal{F}$ to which every entire curve $f: \mathbb{C}\rightarrow X$ must tangent. For establishing the algebraic degeneracy of $f$, McQuillan introduced {\sl Nevanlinna currents} for capturing asymptotic behavior of 
$f$ (see Section~\ref{Nevanlinna-currents}), and the overall strategy of~\cite{McQuillan-1998} is very deep and involved.
Later, Brunella~\cite{Brunella-1999} provided a simplified proof of McQuillan's result, a key
trick being to decompose every obtained Nevanlinna current $\mathsf{T}$ by  Siu's theorem as $\mathsf{T}=\mathsf{T}_{\text{alg}}+\mathsf{T}_{\text{diff}}$  (see Section~\ref{Nevanlinna-currents}).
In passing, Brunella raised the following two conjectures~\cite[page 200]{Brunella-1999}.

\begin{Conj}\label{Brunella conjecture 2}
    For every entire curve $f: \mathbb{C}\rightarrow X$,
    one can choose some increasing radii $\{r_i\}_{i\geqslant 1}\nearrow \infty$ such that
    the obtained Nevanlinna current $\mathsf{T}=\mathsf{T}_{\text{alg}}+\mathsf{T}_{\text{diff}}$ has either trivial  singular part $\mathsf{T}_{\text{alg}}= 0$ or trivial diffuse part $   \mathsf{T}_{\text{diff}}= 0$.
\end{Conj}

\begin{Conj}\label{Brunella conjecture 1}
 Some entire curve shall produce a Nevanlinna current $\mathsf{T}=\mathsf{T}_{\text{alg}}+\mathsf{T}_{\text{diff}}$ with both nontrivial singular part $\mathsf{T}_{\text{alg}}\neq 0$ and nontrivial diffuse part $   \mathsf{T}_{\text{diff}}\neq 0$.
\end{Conj}

There are simplified analogues of Nevanlinna currents, called Ahlfors currents, by which Duval~\cite{Duval-2008} obtained a deep, quantitative Brody Lemma and characterized complex hyperbolicity in terms of linear isoperimetric
inequality for holomorphic discs (see also~\cite{Kleiner}). Using such currents, some geometric refinement~\cite{DH-2018} of the classical Cartan’s
Second Main Theorem, and some higher dimensional analogue~\cite{HV-2021} of Weierstrass-Casorati Theorem were
established. See also~\cite{Dinh-Sibony-2018, Dujardin-ICM} for recent key applications in complex dynamics. It is natural and fundamental to ask

\begin{Ques}\label{uniqueness question}
    Are all Nevanlinna/Ahlfors currents associated with the same entire curve cohomologically equivalent?
\end{Ques}

The conjecture~\ref{Brunella conjecture 2} (analogous to the ``zero--one law'' in probability) is still open,  and the author believes that, by ``Oka principle'',
certain counterexample shall exist.
The conjecture~\ref{Brunella conjecture 1} (answer: yes) and Question~\ref{uniqueness question} (answer: no)
were solved by constructing  exotic examples~\cite{HX-2021}.  The large-scale behaviors of entire curves may have flexibility  rather than rigidity.

\begin{Thm}[\cite{HX-2021}]
\label{theorem HX21}
    There exists an entire curve $f: 
    \mathbb{C}\rightarrow X$ such that, 
    given any cardinality $|I|\in \mathbb{Z}_+\cup \{\infty\}$ and any a priori requirement that $\mathsf{T}_{\text{diff}}$ is trivial or not, by taking certain increasing radii $\{r_i\}_{i\geqslant 1}$ tending to infinity, the sequence of holomorphic discs $\{f\restriction_{\mathbb{D}_{r_i}}\}_{i
    \geqslant 1}$ yields a Nevanlinna/Ahlfors current  $\mathsf{T}=\mathsf{T}_{\text{alg}}+\mathsf{T}_{\text{diff}}$ with the desired shape $\mathsf{T}_{\text{alg}}=\sum_{i\in I} \,\alpha_i\cdot[C_i]$ in Siu's decomposition.
\end{Thm}

In hindsight, the above result reflects Oka property of $X$ (cf. \cite{Oka-book}).

\medskip\noindent
{\bf Convention.} We denote   by $\mathbb{D}(a,r)$ 
(resp. $\overline{\mathbb{D}}(a,r)$)
the open (resp. closed) disc  centered at $a\in \mathbb{C}$ with radius $r$.
When $a=0$, we write $\mathbb{D}_r$  (resp. $\overline{\mathbb{D}}_r$)  for short of
$\mathbb{D}(0,r)$ 
(resp. $\overline{\mathbb{D}}(0,r)$).

\medskip

After receiving an early manuscript of
\cite{HX-2021}, 
Sibony~\cite{Sibony-email} asked

\begin{Ques}\label{Sibony-question}
	Show that for $Y=\mathbb{CP}^n$, or complex torus, all Ahlfors
	 currents can be obtained by single entire curve $f: \mathbb{C}\rightarrow Y$.	
\end{Ques}

Here Sibony might allow us to use any holomorphic disc $\{f\restriction_{ \mathbb{D}(a, r)}\}_{a\in \mathbb{C}, r>0}$ not necessarily centered at the origin.\footnote[1]{\,\,``(...) I suspect, that for $\mathbb{P}^n$, or torus, all the Ahlfors currents can be obtained just using one map, see the discussion on Birkhoff approach~\cite{Birkhoff} in the enclosed paper~\cite{Dinh-Sibony-2020} (...)''\,---\,Sibony~\cite{Sibony-email}} 
He also suggested a hint~\cite{Hormander}. 
Around the same time,  Yau 
asked the following question   during a seminar talk given by the author in Tsinghua.

\begin{Ques}
	\label{Yau-question}
	Can we construct an entire curve $g: \mathbb{C}\rightarrow E_1\times E_2$ in the product of two elliptic curves $E_1, E_2$ such that by using concentric holomorphic discs $\{g\restriction_{ \mathbb{D}_r}\}_{r>0}$ we can produce two Nevanlinna/Ahlfors currents supported in distinct irreducible algebraic curves $\mathcal{C}_1, \mathcal{C}_2$ respectively?
\end{Ques}

The motivation of the above question is wondering whether there exists any  entire curve
$f: \mathbb{C}\rightarrow X$ producing two distinct Nevanlinna/Ahlfors currents 
(with different cohomology classes)
supported on  distinct curves $C_1, C_2\subset X$ 
respectively. 
By a result~\cite{Duval-2006} of Duval,  $C_1$, $C_2$ must be either rational or elliptic.
The simplest such $X$ shall be the product of two  curves with genus $\leqslant 1$.
By using Weierstrass 
$\mathcal{P}$-functions mapping elliptic curves onto $\mathbb{CP}^1$, we only need to consider the case that
 $X=E_1\times E_2$ is the product of two  elliptic curves.

\begin{Def}
A connected complex manifold $Y$ with a complete distance function $\mathsf{d}_Y$ is
said to have the {\sl weak  Oka-1 property}, if for any two disjoint closed discs $D_1, D_2$
contained in a larger closed disc $D\subset \mathbb{C}$,
for
any holomorphic map $F: U\rightarrow Y$ from 
a neighborhood $U$ of $D_1\cup D_2$, for any error bound $\mathsf{e}>0$, there exists some holomorphic
map 
$\hat{F}: \hat{U}\rightarrow Y$ 
defined on some neighborhood $\hat{U}$ of $D$ 
which approximates $F$ on $D_1\cup D_2$ closely 
\[
\mathsf{d}_Y\,
(
\hat{F}(z), F(z)
)\,
\leqslant\,
\mathsf{e}
\qquad
{\scriptstyle(\forall\, z\,\in\, D_1\cup D_2)}.
\]
\end{Def}

Examples of such manifolds $Y$ include all Oka manifolds, e.g., $\mathbb{CP}^n$, complex tori, etc~\cite{Oka-book}. Recently, Alarc\'on and Forstneri\v{c}~\cite{Alarcon-Forstneric}  introduced a larger class of manifolds termed {\em Oka-1}, 
which clearly satisfy the above defined property.
A remarkable result in~\cite{Alarcon-Forstneric}
states that all rationally connected (e.g. Fano) manifolds are Oka-1.

\medskip\noindent
{\bf Theorem A.} 
{\it
Let $Y$  be a compact complex
manifolds satisfying the weak Oka-1 property. Then  there exists an
entire curve $f: \mathbb{C}\rightarrow Y$ generating all Nevanlinna/Ahlfors currents on $Y$
by  holomorphic discs  $\{f\restriction_{ \mathbb{D}(a, r)}\}_{a\in \mathbb{C}, r>0}$.
}

\medskip

Our Algorithm for  producing such $f$ also found applications 
in designing universal holomorphic/meromorphic functions in several variables with slow growth (cf.~\cite{CHX-2023, Guo-Xie-1}), hence solving an open problem~\cite[Problem 9.1]{Dinh-Sibony-2020} raised by Dinh and Sibony.

\medskip\noindent
{\bf Theorem B.} 
{\it
 There exists some entire curve $g: \mathbb{C}\rightarrow X$
in the  product
$X:=E_1\times E_2$ of two elliptic curves $E_1$ and $E_2$, such that the concentric
holomorphic discs  $\{g\restriction_{\mathbb{D}_r}\}_{r>0}$  can generate infinitely many distinct  Nevanlinna/Ahlfors currents
proportional to the extremal currents of integration along curves $[\{e_1\}\times E_2]$,  $[E_1\times \{e_2\}]$  for all $e_1\in E_1, e_2\in E_2$  simultaneously.
}

\medskip

This shows
striking  holomorphic flexibility of entire curves in large scale geometry. 
The phenomenon is new and complementary to the exotic examples of~\cite{HX-2021}.

\begin{Conj}
Let $X$ be a compact Oka manifold.
	There should be some entire curve $g: \mathbb{C}\rightarrow X$ such that  $\{g\restriction_{ \mathbb{D}_r}\}_{r>0}$  generate all Nevanlinna/Ahlfors currents on $X$.
\end{Conj}


\begin{Rmk}
Concerning the Green-Griffiths conjecture  for low degree surfaces in $\mathbb{CP}^3$,
 one challenging problem is whether there exists any
smooth hyperbolic surface with degree $5$.
For degree $6$ and above, such examples exist~\cite{Duval-degree6,  Tuan-IMRN}
 by using Zaidenberg's deformation method~\cite{Zaidenberg-1989}.
 Another key open problem
 in the subject is whether hyperbolicity property preserves along large deformations, e.g., in the complement of certain pluripolar set in  parameter space. The classical Brody lemma only deals with small deformations. The two highlighted  problems shall shed light for the Green-Griffiths conjecture
 for general surfaces in $\mathbb{CP}^3$ with  degrees~$\geqslant 5$.
\end{Rmk}

Here is the structure of this paper. In  Section~\ref{Nevanlinna-currents}, we present the definitions of Nevanlinna/Ahlfors currents. In  Section~\ref{section 3}, we introduce an {\sl Algorithm} for Theorem A.
In Section~\ref{section 4}, 
we construct twisted entire curves to show Theorem B.
The insight comes from Oka theory and complex dynamics.
The construction is based on a sophisticated induction process ({\sl Ping Pong}),  using classical complex analysis in full strength, 
employing notably the Riemann mapping theorem, the little Picard theorem, and Runge's approximation theorem in every step.

\bigskip\noindent
{\bf Acknowledgments.}
 This work is motivated by
insightful questions of Sibony and  Yau, and is inspired by excellent lectures of Forstneri\v{c}.  Some ideas and results of this paper were presented in
 ``Oka Theory and Complex Geometry Conference''
 at Sophus Lie Center (June 2023) organized by Trung Tuyen Truong.  It is a  pleasure for the author to warmly thank Dinh Tuan Huynh for   discussions and collaborations over the years.
 He also thanks Hao Wu (NUS) for  stimulating conversations,
and  Yi C. Huang for careful reading of the manuscript.
  The author is very grateful to the referees for nice suggestions which largely improved the
exposition.   

\medskip\noindent
{\bf Funding.}
This work was supported  by 
National Key R\&D Program of China Grant
No.~2021YFA1003100 and  NSFC Grant No.~12288201.


\section{\bf Nevanlinna and Ahlfors currents}
\label{Nevanlinna-currents}

Let $X$ be a compact complex manifold equipped with a Hermitian  form $\omega$.
Let $f:  \mathbb{D}_R\rightarrow X$ be a nonconstant holomorphic disc, smooth up to the boundary (i.e., $f$ is a restriction of holomorphic map defined on a neighborhood of $\overline{\mathbb{D}}_{R}$). We can associate  with $f$ a  positive current $T_f$ of bidimension $(1,1)$,  which evaluates 
every $\eta$ in 
the set $\mathcal{A}^{(1, 1)}(X)$ of smooth $(1,1)$–forms on $X$ by 
\[
T_f(\eta):=
\int_{0}^R\,\frac{\text{d}t}{t}\int_{\mathbb{D}_t}\,
f^*\eta.
\]
The reason of such definition roots in Jensen's formula.
Set
\[
L_f(\omega)
:=
\int_0^R\,
\text{Length}_{\omega}\,\big(f(\partial \mathbb{D}_t)\big)\,
\frac{\text{d}t}{t}.
\]
Let $\{f_i:  \mathbb{D}_{r_i}\rightarrow X\}_{i\geqslant 1}$ 
be a sequence of nonconstant holomorphic discs smooth up to the
	boundary, such that the {\sl length-area condition} holds
\begin{equation}
	\label{length-area condition}
	\lim_{i\rightarrow \infty}	
	\dfrac{L_{f_i}(\omega)}{T_{f_i}(\omega)}
	=
	0.
\end{equation}
 Consider the family of positive
currents of bounded mass $\{\Phi_{f_i}\}_{i\geqslant 1}$ defined as
\begin{equation}
    \label{define Phi}
\Phi_{f_i}(\eta):=
\dfrac{T_{f_i}(\eta)}{T_{f_i}(\omega)}
\qquad
{\scriptstyle(\forall\,i\,\geqslant\, 1;\,\,\forall\, \eta\,\in \mathcal{A}^{(1, 1)}(X))}.
\end{equation}
Then 
by Banach–Alaoglu’s theorem, some subsequence 
$\{\Phi_{f_{k_\ell}}\}_{\ell\geqslant 1}$
converges in weak topology to a positive current
$\Phi$. The condition~\eqref{length-area condition}
guarantees that $\Phi$ is in fact closed (cf.~\cite[page~55]{Huynh-PhD}). Such  a  limit current is called Nevanlinna current.

In particular, given an entire curve $f: \mathbb{C}\rightarrow X$, by Ahlfors lemma~(cf.~\cite[page~55]{Huynh-PhD}), one 
can find some increasing radii $\{r_i\}_{i\geqslant 1}$ tending to infinity, such that the sequence of restricted holomorphic discs $\{f\restriction_{\mathbb{D}_{r_i}}\}_{i\geqslant 1}$
satisfies the length-area condition~\eqref{length-area condition}.
Thus we can receive some Nevanlinna current associated with the entire curve $f$. 

The definition of Ahlfors currents
is likewise and simpler. Again we start with
 a sequence of nonconstant holomorphic discs $f_i:  \mathbb{D}_{r_i}\rightarrow X$ {\em smooth up to the
boundary}, satisfying another {\sl length-area condition} that
\begin{equation}
	\label{length-area-condition-Ahlfors}
\lim_{i\rightarrow \infty}\,	\dfrac{\text{Length}_{\omega}\big(f_i(\partial \mathbb{D}_{r_i})\big)}{\text{Area}_{\omega}\big(f_i( \mathbb{D}_{r_i})\big)}
=
 0.
\end{equation}
From the sequence of normalized currents
\begin{equation}
\label{Ahlfors currents}
\Big\{
	\frac{[f_i(\mathbb{D}_{r_i})]}{\text{Area}_{\omega}\big(f_i( \mathbb{D}_{r_i})\big)}
\Big\}_{i\geqslant 1}
\end{equation} of bounded mass,
by compactness and diagonal argument, 
after passing to some subsequence, in the limit 
one receives some positive  bidimension $(1, 1)$ {\sl Ahlfors current},   which is in fact closed because of the condition~\eqref{length-area-condition-Ahlfors}. 

In particular, given an entire curve $f: \mathbb{C}\rightarrow X$, by Ahlfors lemma (cf.~\cite{Duval-2017}), we
can find some increasing radii $\{r_i\}_{i\geqslant 1}\nearrow \infty$, such that the  sequence of restricted holomorphic discs $\{f\restriction_{\mathbb{D}_{r_i}}\}_{i\geqslant 1}$
satisfies the length-area condition~\eqref{length-area-condition-Ahlfors}.
Thus $f$ produces some Ahlfors current.

One advantage of
Nevanlinna currents 
over Ahlfors currents is that, the associated Nevanlinna currents
with an entire curve $f: \mathbb{C}\rightarrow X$ are always nef provided that $f$ is algebraically nondegenerate (cf.~\cite{McQuillan-ICM}).

Since every Nevanlinna/Ahlfors current $\mathsf{T}$ is positive and closed, 
by Siu's decomposition theorem~\cite{Siu-1974}, 
$\mathsf{T}=\mathsf{T}_{\text{alg}}+\mathsf{T}_{\text{diff}}$, 
where the singular part $\mathsf{T}_{\text{alg}}=\sum_{i\in I} \,\alpha_i\cdot[C_i]$
is some positive linear combination ($\alpha_i>0$; $I\subset \mathbb{Z}_+$, could be $\emptyset$) of currents of integration on distinct irreducible
algebraic curves $C_i$, and where the diffuse part $\mathsf{T}_{\text{diff}}$ is a positive closed $(1, 1)$–current having zero Lelong
number along any algebraic curve. A result~\cite{Duval-2006} of Duval shows that, if the singular part $\mathsf{T}_{\text{alg}}$ is nontrivial, then every irreducible component $C_i$ in the support of $\mathsf{T}_{\text{alg}}$  must be rational or elliptic.

\section{\bf Patching together infinite disjoint holomorphic discs} 
\label{section 3}
	Throughout this section, we fix a Hermitian form $\omega$ on $Y$, which induces a distance $\mathsf{d}$.  
The proof of Theorem~A
relies on the following two observations.

\begin{Obs}
\label{Observation 1}
   We can select some countable holomorphic discs $\{f_i: \mathbb{D}_{r_i}\rightarrow Y\}_{i\geqslant 1}$ smooth up to the boundary,  from which all Nevanlinna/Ahlfors currents on $Y$ can be generated.
\end{Obs}

\begin{proof}
	For every $r>0$,
	denote by $\mathsf{Hol}(\mathbb{D}_r, Y)$  the set of holomorphic discs $f: \mathbb{D}_r\rightarrow Y$.
	The space  $\mathsf{Hol}(\mathbb{D}_r, Y)$ has a natural distance $\mathsf{d}_{r}$ which associates  $f, g: \mathbb{D}_r\rightarrow Y$ with the supremum deviation
	$$
	\mathsf{d}_{r}(f, g)
	:=
	\sup_{z\in \mathbb{D}_r}\, \mathsf{d}(f(z), g(z)).$$
	
	For any $m>0$,
	denote by $\mathsf{Hol}(\mathbb{D}_r, X)_{m}$ the subset of $\mathsf{Hol}(\mathbb{D}_r, Y)$ defined by  the uniform  derivative  bound $||\text{d}f||_{\omega}\leqslant m$.
	For any sequence of holomorphic discs $\{f_i\}_{i\geqslant 1}$ in $\mathsf{Hol}(\mathbb{D}_r, X)_{m}$,
	by mimicking the proof of Montel's theorem about normal families, i.e., by using the Arzel\`a-Ascoli theorem and Cauchy’s integral formula, 
	we can find some  subsequence $\{f_{n_i}\}_{i\geqslant 1}$  uniformly convergent to certain continuous map $f_{\infty}: \mathbb{D}_r\rightarrow Y$ which is in fact an element in $\mathsf{Hol}(\mathbb{D}_r, Y)_{m}$. 
	Summarizing, any infinite sequence of elements in $\mathsf{Hol}(\mathbb{D}_r, Y)_{m}$ has an accumulation point. Therefore, $\mathsf{Hol}(\mathbb{D}_r, Y)_{m}$ is compact for every $r, m>0$.
	Hence we can select a dense countable subset $\mathcal{D}_{r, m}$ of $\mathsf{Hol}(\mathbb{D}_r, X)_{ m}$ with respect to the distance $\mathsf{d}_{r}$.
	
	We try to use the union
	$$
	\mathcal{D}
	:=
	\cup_{r\in \mathbb{Q}_+,\, m\in \mathbb{Z}_+}\, \mathcal{D}_{r, m}
	$$
	as our desired family of countable holomorphic discs. 
	However, there is one  subtly, that an element
	$f: \mathbb{D}_r\rightarrow Y$ in $\mathcal{D}$ might not be smooth up to the boundary (i.e., $f$ is a restriction of holomorphic map defined on a neighborhood of $\overline{\mathbb{D}}_{r}$). The trick to remedy this is by replacing each $f\in \mathcal{D}$ by countably  holomorphic discs $\{f_{\ell}\}_{\ell\geqslant 1}$ obtained by  scaling
	\[
	f_{\ell}(z)
	:=
	f(\epsilon_\ell \cdot z)
	\qquad
{\scriptstyle(\ell\,\geqslant\, 1;\,\,\forall\, z\,\in \mathbb{D}_r)},
	\]
	where $\{\epsilon_\ell\}_{\ell\geqslant 1}\nearrow 1$, say $\epsilon_{\ell}=1-1/2^\ell$.
	The {\em new} $\mathcal{D}$, which remains countable,  suffices for our purpose.
	
The remaining argument goes as follows.
For any holomorphic disc
$f:  \mathbb{D}_{r}\rightarrow Y$ smooth up to the boundary, automatically
$||\text{d}f||_{\omega}$ is uniformly bounded from above.
The reparametrization of $f$ 
\[
f_{1}(\bullet):=
f(r\cdot \bullet)
\,
:\,
\mathbb{D}_{1}
\rightarrow
Y
\]
by scaling gives the same current 
$\Phi_{f_1}=\Phi_{f}$
by~\eqref{define Phi}.
Since $||\text{d}f_1||_{\omega}< \infty$,
${f_1}$
can be approximated by a sequence 
$\{f_{1, j}\}_{j\geqslant 1}$ in $\mathcal{D}$
 with respect to $\mathsf{d}_1$. 
Thus
 $\Phi_{f_1}=
 \lim_{j\geqslant 1}\, \Phi_{f_{1, j}}$.
This concludes the proof.
\end{proof}

\begin{Obs}
\label{Observation 2}
  Given countable holomorphic discs $\{f_i: \mathbb{D}_{r_i}\rightarrow Y\}_{i\geqslant 1}$ smooth up to the boundary,  we can construct
  an entire curve $f: \mathbb{C}\rightarrow Y$ such that
 every $f_i$ can be approximated at any precision by  holomorphic discs 
  of $f$, i.e., there exists some sequence $\{c_{i, \ell}\}_{\ell\geqslant 1}$ in $\mathbb{C}$ such that
  \[
  \lim_{\ell\rightarrow \infty}
  f(z+c_{i, \ell})
  =
  f_i(z)
  \qquad
  {\scriptstyle(\forall\, i\,\geqslant\, 1)}
  \]
  for all
  $z\in \mathbb{D}_{r_i}$
  uniformly.
\end{Obs}

\begin{proof}
Fix a  bijection
$
\varphi=(\varphi_1, \varphi_2)\,
:\,
\mathbb{Z}_+
\overset{\sim}{\longrightarrow}
\mathbb{Z}_+ \times \mathbb{Z}_+$.
Take a sequence of shrinking  positive numbers $\{\epsilon_i\}_{i\geqslant 1}$, say $\epsilon_i=1/2^i$,  such that 
$
\lim_{j\rightarrow \infty} \sum_{i\geqslant j} \epsilon_i=0$. 
We 
make another sequence of holomorphic discs $\{g_j : \mathbb{D}_{\hat{r}_j}\rightarrow Y
\}_{j\geqslant 1}$
such that each $f_i$ repeats infinitely many times in $\{g_j\}_{j\geqslant 1}$,
for instance
we can take
$
g_j:=f_{\varphi_1(j)}$, 
$
	\hat{r}_j
	:=
	r_{\varphi_1(j)}$. 
Now we construct a desired entire curve $f$ by taking the limit of some convergent holomorphic discs
$
\{F_i: \mathbb{D}_{R_i}\rightarrow Y\}_{i\geqslant 1}$ as
$
R_i \nearrow \infty$.

\smallskip

In {\em Step $1$}, we set three key data $F_1:=g_1$, $R_1:=\hat{r}_1$, $c_1:=0$.

Subsequently, in {\em Step $i+1$} for $i=1, 2, 3, \dots$,
we choose a center $c_{i+1}\in \mathbb{C}$ far away from the origin, so that
the two discs
$\overline{\mathbb{D}}_{R_i}$ and
$\overline{\mathbb{D}}(c_{i+1}, \hat{r}_{i+1})$ 
stay apart.
We define  $\hat{g}_{i+1}(\bullet):=g_{i+1}(\bullet-c_{i+1})$ on a neighborhood of $\overline{\mathbb{D}}(c_{i+1}, \hat{r}_{i+1})$ 
by translation of $g_{i+1}$. Now
we take a large disc $\overline{\mathbb{D}}_{R_{i+1}}$ 
having radius 
$
R_{i+1}>|c_{i+1}|+\hat{r}_{i+1}$, so that it 
contains
$\overline{\mathbb{D}}_{R_i}$ and $\overline{\mathbb{D}}(c_{i+1}, \hat{r}_{i+1})$ in the interior.
By the assumed weak Oka-1 property of $Y$,
we can ``extend'' $F_i$ and $\hat{g}_{i+1}$ to some holomorphic disc $F_{i+1}$ 
defined in a neighborhood of $\overline{\mathbb{D}}_{R_{i+1}}$ such that
\begin{align}
	\label{small errors}
	\sup_{z\in \overline{\mathbb{D}}_{R_i}}\,\mathsf{d}(F_{i+1}(z), F_i(z))
	\leqslant \epsilon_{i+1},
	\qquad
	\sup_{z\in \overline{\mathbb{D}}(c_{i+1}, \hat{r}_{i+1})}
	\mathsf{d}(F_{i+1}(z), \hat{g}_{i+1}(z))
	\leqslant \epsilon_{i+1}.
\end{align}
Since $
\lim_{j\rightarrow \infty} \sum_{i\geqslant j} \epsilon_i=0$, it is clear that
$\{F_i\}_{i\geqslant 1}$
converges to an
entire curve
$f: \mathbb{C}\rightarrow Y$. Moreover,  by using~\eqref{small errors} repeatedly, we see that $f$  approximates each $\hat{g}_{i+1}$ well

\begin{align*}
	\sup_{z\in \overline{\mathbb{D}}(c_{i+1}, \hat{r}_{i+1})}
	\mathsf{d}(f(z), \hat{g}_{i+1}(z))
	&
	\leqslant
	\sup_{z\in \overline{\mathbb{D}}(c_{i+1}, \hat{r}_{i+1})}
\mathsf{d}(F_{i+1}(z), \hat{g}_{i+1}(z))
	+
	\sum_{j\geqslant i+1}
	\sup_{z\in \overline{\mathbb{D}}_{R_j}}\,\mathsf{d}(F_{j+1}(z), F_j(z))
	\nonumber
	\\
	&
	\leqslant
	\epsilon_{i+1}+
	\sum_{j\geqslant i+1}
	\epsilon_{j+1}
	\,\,
	\rightarrow
	0
	\qquad
	{\scriptstyle
		(\text{as } i\, \rightarrow\, \infty)}.
\end{align*}
In other words, 
$\sup_{
z\in
\overline{\mathbb{D}}_{\hat{r}_{i+1}}
}\,
\mathsf{d}
\big(
f(c_{i+1}+ z), g_{i+1}(z)
\big)
\rightarrow
0$ 	as
$i\rightarrow\, \infty$. 
This finishes the proof.
\end{proof}

The above  Algorithm for 
constructing $f$  also plays a
 key r\^ole
in the solutions~\cite{CHX-2023, Guo-Xie-1} to an open problem~\cite[Problem 9.1]{Dinh-Sibony-2020} raised by Dinh and Sibony.

\medskip\noindent
{\bf Proof of Theorem A.}
Combining Observations~\ref{Observation 1} and~\ref{Observation 2}, 
 Theorem~A follows
  directly. 
  \qed

\begin{Rmk}
In fact, a moment of reflection shows that,  the obtained entire curve $f: \mathbb{C}\rightarrow Y$ in the above construction 
  is actually universal in the sense of~\cite{Kusakabe-2017, Guo-Xie-1}.
\end{Rmk}

\begin{Ques}
Is there a compact complex manifold with the weak Oka-1 property but is not Oka-1 (cf.~\cite{Alarcon-Forstneric})?
\end{Ques}

\section{\bf Twisted entire curves in the product of two elliptic curves}
\label{section 4}

Given two elliptic curves $E_1=\mathbb{C}/\Gamma_1$ and $E_2=\mathbb{C}/\Gamma_2$ where $\Gamma_1, \Gamma_2$ are two lattices in $\mathbb{C}$. 
Let  $\pi_1: \mathbb{C}\rightarrow E_1$,  $\pi_2: \mathbb{C}\rightarrow E_2$ and $\pi: \mathbb{C}\times \mathbb{C}\rightarrow E_1\times E_2=X$ be the
canonical projections. 
The standard Euclidean norm $||\bullet||_{\mathbb{C}^2}$ on $\mathbb{C}^2$ corresponds to the positive $(1, 1)$-form 
$\frac{\sqrt{-1}}{2}(\text{d}z_1\wedge
\text{d}\bar{z}_1+\text{d}z_2\wedge
\text{d}\bar{z}_2)$, where
 $(z_1, z_2)$ are the usual coordinate functions  of $\mathbb{C}^2$.  Descending via $\pi$,  
we receive a reference Hermitian form $\omega$
 on $X$.

For every $i=1, 2$,
we fix some countable everywhere dense sequence
$\{e_{i,\ell}\}_{\ell \geqslant 1}\subset E_i$.  
Our desired entire curve will write as $g:=\pi\circ F$, where $F=(G_1, G_2): \mathbb{C}\rightarrow \mathbb{C}\times \mathbb{C}$
will be obtained 
by taking the limit of some convergent holomorphic discs 
\begin{equation}
\label{holo discs are convergent}
F_\ell=(G_{1, \ell}, G_{2, \ell}): \mathbb{D}_{R_{\ell}}\rightarrow \mathbb{C}\times \mathbb{C},
\quad
R_\ell \nearrow \infty
\qquad
{\scriptstyle
	(\ell\,=\,0,\,1,\, 2,\, \dots)}.
\end{equation}

In {\em Step $0$}, we set $R_0=1$, and choose two nonconstant holomorphic functions $G_{1, 0}$, $G_{2, 0}$ defined on a neighborhood $V_0$ of the unit disc $\mathbb{D}_1$.

Subsequently, in {\em Step $2k-1$}
and  {\em Step $2k$} for $k=1, 2, 3, \dots$, 
using Runge's approximation theorem~(cf. e.g.~\cite[page~94]{Gaier}) upon the holomorphic functions $G_{1, 2k-2}$ and $G_{2, 2k-2}$ defined on $V_{{2k-2}}\supset \overline{\mathbb{D}}_{R_{2k-2}}$, for both $i=1, 2$, we approximate
$G_{i, 2k-2}$ on
a smaller neighborhood $V_{2k-2}'\subset\subset V_{2k-2}$ (i.e.,  relatively compact)
of 
$\overline{\mathbb{D}}_{R_{2k-2}}$
by some nonconstant entire functions
$\hat{G}_{i, 2k-2}$ within small error $\epsilon_{2k-2}/2023$,
so that 
$\hat{F}_{2k-2}:=(\hat{G}_{1, 2k-2}, \hat{G}_{2, 2k-2})$ is very close to $F_{2k-2}$ on $V_{2k-2}'$
\[
||\hat{F}_{2k-2}(z)-F_{2k-2}(z)||_{\mathbb{C}^2}\leqslant 2\epsilon_{2k-2}/2023
\qquad
{\scriptstyle
	(\forall\,z\,\in\,V_{2k-2}')}.
\]

\medskip\noindent
{\bf Plan of Step $2k-1$}.
To choose some large radius $R_{2k-1}>R_{2k-2}+1$ and then to modify
$\hat{F}_{2k-2}$ as
${F}_{2k-1}
:=(\hat{G}_{1, 2k-2}, \hat{G}_{2, 2k-2}\cdot H_{2k-1})$
by some auxiliary holomorphic
function $H_{2k-1}$ defined on a neighborhood of $\overline{\mathbb{D}}_{R_{2k-1}}$,
such that
${F}_{2k-1}$ is very close to $\hat{F}_{2k-2}$ (hence ${F}_{2k-2}$) on $V_{2k-2}'$
\begin{equation}
\label{small error 2023}
||{F}_{2k-1}(z)-\hat{F}_{2k-2}(z)||_{\mathbb{C}^2}
\leqslant \frac{\epsilon_{2k-2}}{2023},\,
||{F}_{2k-1}(z)-{F}_{2k-2}(z)||_{\mathbb{C}^2}
\leqslant \frac{3\epsilon_{2k-2}}{2023}
\qquad
{\scriptstyle
	(\forall \,z\ \in\ V_{2k-2}')},
\end{equation}
while
the induced current $\Phi_{\pi\circ {F}_{2k-1}\restriction_{\mathbb{D}_{R_{2k-1}}}}$ for constructing Nevanlinna currents (see~\eqref{define Phi}) 
and the current
$\frac{[\pi\circ {F}_{2k-1}(\mathbb{D}_{R_{2k-1}})]}{\text{Area}(\pi\circ {F}_{2k-1}(\mathbb{D}_{R_{2k-1}}))}$
for constructing Ahlfors currents (see~\eqref{Ahlfors currents}) both
have mass concentrated near $\{e_{1, k}\}\times E_2$ in a quantitative sense to be specified later.

\medskip\noindent
{\bf Observation 1.} 
By the little Picard theorem,
the range of the nonconstant entire function
$\hat{G}_{1, 2k-2}$  cannot avoid the set
$\pi_1^{-1}(e_{1, k})\setminus \hat{G}_{1, 2k-2}(\overline{\mathbb{D}}_{R_{2k-2}+1})$ which contains infinitely many points!
Hence we can find some  $z_{2k-1}\in \mathbb{C}\setminus \overline{\mathbb{D}}_{R_{2k-2}+1}$
with $\hat{G}_{1, 2k-2}(z_{2k-1})\in \pi_1^{-1}(e_{1, k})$.
The upshot is that
\[
\pi\circ {F}_{2k-1}(z_{2k-1})
=
\big
(\pi_1\circ\hat{G}_{1, 2k-2}(z_{2k-1}),
\pi_2\circ(\hat{G}_{2, 2k-2}\cdot H_{2k-1})(z_{2k-1})
\big)
=
(e_{1, k}, *)
\]lies in the curve $\{e_{1, k}\}\times E_2$
no matter how we choose
$H_{2k-1}$. 

\medskip\noindent
{\bf Idea.}
Fix a very small neighborhood $U_{2k-1}$ of $z_{2k-1}$ such that
$\pi_1\circ\hat{G}_{1, 2k-2}$ takes value
in a small neighborhood of $\pi_1\circ\hat{G}_{1, 2k-2}(z_{2k-1})=e_{1, k}$,
say 
\[
\pi_1\circ\hat{G}_{1, 2k-2}(U_{2k-1})\,
\subset\,
\mathbb{D}(e_{1, k}, 2^{-(2k-1)-2023}):=
\pi_1
\big(
\mathbb{D}(\hat{G}_{1, 2k-2}(z_{2k-1}), 2^{-(2k-1)-2023})
\big).
\]
Now we demand that
$H_{2k-1}$ is almost $1$
on $\overline{\mathbb{D}}_{R_{2k-1}}\setminus U_{2k-1}$ whereas
$H_{2k-1}$ oscillates violently
on $U_{2k-1}$ so that the currents
$\Phi_{\pi\circ {F}_{2k-1}\restriction_{\mathbb{D}_{R_{2k-1}}}}$ and $\frac{[\pi\circ {F}_{2k-1}(\mathbb{D}_{R_{2k-1}})]}{\text{Area}(\pi\circ {F}_{2k-1}(\mathbb{D}_{R_{2k-1}}))}$ 
have most mass near $\{e_{1, k}\}\times E_2$
in a quantitative sense to be specified later.

\medskip

\medskip\noindent
{\bf Observation 2.} 
Set $R_{2k-1}:=|z_{2k-1}|>R_{2k-2}+1$.
If we can find some
holomorphic function
$\mathcal{X}_{2k-1}$ defined on a neighborhood $V_{2k-1}$ of 
$\overline{\mathbb{D}}_{R_{2k-1}}$
such that
$|\mathcal{X}_{2k-1}|<1$ on $\overline{\mathbb{D}}_{R_{2k-1}}\setminus U_{2k-1}$ whereas
$|\mathcal{X}_{2k-1}(z_{2k-1})|>1$, then 
$H_{2k-1}:=1+\mathcal{X}_{2k-1}^{M_{2k-1}}$ for some very large integer $M_{2k-1}\gg 1$ will suffice for the goal of  {\em Step $2k-1$}. 

\medskip
The existence of such $\mathcal{X}_{2k-1}$ is guaranteed by the following

\medskip\noindent
{\bf Key Lemma.} 
{
\it 
Let $z_0\in \partial \mathbb{D}_R$, and let $U$ be a neighborhood of $z_0$. Then one can find some holomorphic function $\mathcal{X}$ defined in a neighborhood of $\overline{\mathbb{D}}_R$ such that $|\mathcal{X}|<1$ on $\overline{\mathbb{D}}_R\setminus U$ while $|\mathcal{X}(z_0)|>1$.
}

\medskip
The proof will be postponed to the  end of this section.
Now we check the  conditions~\eqref{length-area condition} and~\eqref{length-area-condition-Ahlfors}.

\medskip
\noindent
{\bf Length-Area  Estimate (I).} 
Fix $k\in \mathbb{Z}_+$.
For the family of
holomorphic discs 
$$\overline{F}_{2k-1}:=\pi\circ {F}_{2k-1} : {\mathbb{D}_{R_{2k-1}}}\longrightarrow X$$
parameterized by 
$M_{2k-1}\in \mathbb{Z}_+$,
one has
\begin{equation}
    \label{key estimate}
	\dfrac{L_{\overline{F}_{2k-1}}(\omega)}{T_{\overline{F}_{2k-1}}(\omega)}
	\rightarrow
	0
	\qquad
	\text{as $M_{2k-1}$ tends to infinity}.
\end{equation}

\begin{proof}
Denote by $m_{2k-1}>1$ the maximum modulus of 
$|\mathcal{X}_{2k-1}|$ over $\overline{\mathbb{D}}_{R_{2k-1}}$.
The set
\begin{equation}
\label{the tricky small disck}
\{
z\in \mathbb{D}_{R_{2k-1}}
\,:\,
1<
m_{2k-1}^{2/3}
<
|\mathcal{X}_{2k-1}(z)|
<
m_{2k-1},\,
\hat{G}_{2, 2k-2}\cdot \mathcal{X}_{2k-1}'(z)\neq 0
\}
\end{equation}
is clearly open and nonempty.  Take a small closed disc $\overline{\mathbb{D}}(c_{2k-1}, r_{2k-1})$ in it.
Let $\widetilde{m}_{2k-1}>0$
be the minimum modulus of $|\hat{G}_{2, 2k-2}\cdot \mathcal{X}_{2k-1}'|$
on $\overline{\mathbb{D}}(c_{2k-1}, r_{2k-1})$.
Now we compute the derivative
\begin{align}
\nonumber
(\hat{G}_{2, 2k-2}\cdot H_{2k-1})'
&
=
\big(\hat{G}_{2, 2k-2}\cdot (1+\mathcal{X}_{2k-1}^{M_{2k-1}})\big)'
\\
\label{compute derivative}
&
=
M_{2k-1}\cdot
 \mathcal{X}_{2k-1}^{M_{2k-1}-1}\cdot
 \hat{G}_{2, 2k-2}\cdot \mathcal{X}_{2k-1}'
 \,
+\,
\hat{G}_{2, 2k-2}'\cdot
(1+ \mathcal{X}_{2k-1}^{M_{2k-1}}).
\end{align}

Using the elementary inequality
$
|A+B|^2
\geqslant 
\frac{1}{4}\cdot
|A|^2
-
|B|^2
$
for all complex numbers
$A,  B$, 
we can estimate  $\overline{F}_{2k-1}^*\,\omega\geqslant
\overline{F}_{2k-1}^*(\frac{\sqrt{-1}}{2}\text{d}z_2\wedge
\text{d}\bar{z}_2)
$ on
$\overline{\mathbb{D}}(c_{2k-1}, r_{2k-1})$
from below by
\begin{align}
{}\,&
\Big(
\frac{1}{4}\cdot |M_{2k-1}\cdot
 \mathcal{X}_{2k-1}^{M_{2k-1}-1}\cdot
 \hat{G}_{2, 2k-2}\cdot \mathcal{X}_{2k-1}'|^2
 -
 |\hat{G}_{2, 2k-2}'\cdot
(1+ \mathcal{X}_{2k-1}^{M_{2k-1}})|^2
\Big)
\cdot
\frac{\sqrt{-1}}{2}\text{d}z\wedge
\text{d}\bar{z}
\nonumber
\\
\geqslant\,&\,
\Big(
O^+(1)\cdot
 M_{2k-1}^2
\cdot
 |\mathcal{X}_{2k-1}|^{2\cdot(M_{2k-1}-1)}
 -
O^+(1)
  \cdot
  \big(
  1+
|\mathcal{X}_{2k-1}|^{ M_{2k-1}}
\big)^2
\Big)
\cdot
\frac{\sqrt{-1}}{2}\text{d}z\wedge
\text{d}\bar{z}
\nonumber
\\
\text{[$\spadesuit$]}
\quad
\geqslant\,&\,
\Big(
O^+(1)\cdot
 M_{2k-1}^2
\cdot
 |\mathcal{X}_{2k-1}|^{2\cdot(M_{2k-1}-1)}
 -
  O^+(1)
  \cdot
 4
|\mathcal{X}_{2k-1}|^{ 2M_{2k-1}}
\Big)
\cdot
\frac{\sqrt{-1}}{2}\text{d}z\wedge
\text{d}\bar{z}
\nonumber
\\
=\,&\,
\big(
O^{+}(1)
\cdot M_{2k-1}^2
 -
 O^+(1)\cdot 4|\mathcal{X}_{2k-1}|^2
  \big)
\cdot
|\mathcal{X}_{2k-1}|^{2\cdot(M_{2k-1}-1)}
\cdot
\frac{\sqrt{-1}}{2}\text{d}z\wedge
\text{d}\bar{z}
\nonumber
\\
\label{estimate omega from below}
\text{[see~\eqref{the tricky small disck}]}
\quad
\geqslant\,&\,
\big(
O^{+}(1)
\cdot M_{2k-1}^2
 -
 O^+(1)
  \big)
\cdot
({m}_{2k-1}^{2/3})^{ 2\cdot(M_{2k-1}-1)}
\cdot
\frac{\sqrt{-1}}{2}\text{d}z\wedge
\text{d}\bar{z},
\end{align}
where we always abuse the notation
$O^+(1)$ 
for different  positive bounded constants independent of $M_{2k-1}$ (but can depend on $m_{2k-1}$, $\widetilde{m}_{2k-1}$,    $\mathcal{X}_{2k-1}$, 
$\hat{G}_{2, 2k-2}$, $\hat{G}_{1, 2k-2}$,
$R_{2k-1}$,
$c_{2k-1}$,
$r_{2k-1}$),
and where $[\spadesuit]$
uses the fact that $|\mathcal{X}_{2k-1}|>1$
on $\overline{\mathbb{D}}(c_{2k-1}, r_{2k-1})$ due to~\eqref{the tricky small disck}. 
Thus we can bound the area growth
 from below,  without using Jensen's formula, by
\begin{align}
\nonumber
T_{\overline{F}_{2k-1}}(\omega)
&
=
\int_{0}^{R_{2k-1}}\,\frac{\text{d}t}{t}\int_{\mathbb{D}_t}\,
f^*\omega
\\
\nonumber
&
\geqslant
\int_{|c_{2k-1}|+r_{2k-1}}^{R_{2k-1}}\,\frac{\text{d}t}{t}\int_{\overline{\mathbb{D}}(c_{2k-1}, r_{2k-1})}\,
f^*\omega
\quad\text{[since $\overline{\mathbb{D}}(c_{2k-1}, r_{2k-1})
\subset \overline{\mathbb{D}}_t$ for $t\geqslant |c_{2k-1}|+r_{2k-1}$]}
\\
&
\label{estimate T}
\geqslant
\big(
O^{+}(1)
\cdot M_{2k-1}^2
 -
 O^+(1)
  \big)
\cdot
({m}_{2k-1}^{2/3})^{ 2(M_{2k-1}-1)}
\cdot
O^+(1)
\qquad\text{[see~\eqref{estimate omega from below}]}.
\end{align}
We emphasize a key point here that the exponent of $m_{2k-1}>1$ is
\[
\frac{2}{3}\cdot 2(M_{2k-1}-1)\gg M_{2k-1}
\]
for $M_{2k-1}\gg 1$.

Next, on  $\overline{\mathbb{D}}_{R_{2k-1}}$
we have
$
|(\hat{G}_{2, 2k-2}\cdot H_{2k-1})'|
\leqslant
 O^+(1)
\cdot
 M_{2k-1}\cdot
 m_{2k-1}^{M_{2k-1}}$
 by~\eqref{compute derivative}.
Thus
\begin{equation}
\label{estimate omega from above}
    \overline{F}_{2k-1}^*\,\omega \leqslant
\big(
O^+(1)
\cdot
 M_{2k-1}^2\cdot
 m_{2k-1}^{2\cdot M_{2k-1}}
 +
 O^+(1)
 \big)
 \cdot
\frac{\sqrt{-1}}{2}\text{d}z\wedge
\text{d}\bar{z}
\end{equation}
on $\overline{\mathbb{D}}_{R_{2k-1}}$. 
 Therefore we can estimate
 the length growth from above by
 \begin{align}
\nonumber
 L_{\overline{F}_{2k-1}}(\omega)\,
 &
 \leqslant\,
 \big(
O^+(1)
\cdot
 M_{2k-1}^2\cdot
 m_{2k-1}^{2\cdot M_{2k-1}}
 +
 O^+(1)
 \big)^{1/2}
 \cdot
 \int_0^{R_{2k-1}}\,
\text{Length}_{\mathbb{C}}\,(\partial \mathbb{D}_t)\,
\frac{\text{d}t}{t}
\\
\nonumber
 &
 =\,
 O^+(1)
 \cdot
 \big(
O^+(1)
\cdot
 M_{2k-1}^2\cdot
 m_{2k-1}^{2\cdot M_{2k-1}}
 +
 O^+(1)
 \big)^{1/2}
 \cdot 
 2\pi\,R_{2k-1}
 \\
 &
 \label{estimate L}
 \leqslant\,
 O^+(1)
\cdot
 \big(
O^+(1)
\cdot
 M_{2k-1}\cdot
 m_{2k-1}^{M_{2k-1}}
 +
 O^+(1)
 \big).
\end{align}

Hence the desired estimate~\eqref{key estimate}
 follows directly
by comparing the asymptotic growth rates of~\eqref{estimate L} and~\eqref{estimate T}  with respect to the parameter $M_{2k-1}\gg 1$.
\end{proof}

\medskip
\noindent
{\bf Length-Area  Estimate (II).} 
Fix $k\in \mathbb{Z}_+$.
For the family of
holomorphic discs 
$$\overline{F}_{2k-1}:=\pi\circ {F}_{2k-1} : {\mathbb{D}_{R_{2k-1}}}\longrightarrow X$$
parameterized by 
$M_{2k-1}\in \mathbb{Z}_+$,
one has
\begin{equation}
    \label{length-area estimate 2}
    \dfrac{\text{Length}_{\omega}\big(\overline{F}_{2k-1}(\partial \mathbb{D}_{R_{2k-1}})\big)}{\text{Area}_{\omega}\big(\overline{F}_{2k-1}( \mathbb{D}_{R_{2k-1}})\big)}
	\rightarrow
	0
	\qquad
	\text{as $M_{2k-1}$ tends to infinity}.
\end{equation}

\begin{proof}
Following the preceding argument,
by~\eqref{estimate omega from above},
we receive
\[
\text{Length}_{\omega}\big(\overline{F}_{2k-1}(\partial \mathbb{D}_{R_{2k-1}})\big)
\leqslant
\big(
O^+(1)
\cdot
 M_{2k-1}^2\cdot
 m_{2k-1}^{2\cdot M_{2k-1}}
 +
 O^+(1)
 \big)^{1/2}
 \cdot
 2\pi\, R_{2k-1}.
\]
By~\eqref{estimate omega from below}, we have
\begin{equation}
    \label{bound holo disc are on large R}
\text{Area}_{\omega}\big(\overline{F}_{2k-1}( \mathbb{D}_{R_{2k-1}})\big)
\geqslant
\big(
O^{+}(1)
\cdot M_{2k-1}^2
 -
 O^+(1)
  \big)
\cdot
({m}_{2k-1}^{2/3})^{ 2\cdot(M_{2k-1}-1)}
\cdot
\pi\,r_{2k-1}^2.
\end{equation}

Thus the desired estimate 
 follows directly
by comparing the asymptotic growth rates in the above two estimates with respect to the parameter $M_{2k-1}\nearrow \infty$.
\end{proof}

Now we move on to {\em Step $2k$}. Copy
\[
(G_{1, 2k-1}, G_{2,  2k-1}):=
F_{2k-1}
=
(\hat{G}_{1, 2k-2}, \hat{G}_{2, 2k-2}\cdot H_{2k-1})
\]
defined on a neighborhood $V_{2k-1}$ of $\overline{\mathbb{D}}_{R_{2k-1}}$. 
 By Runge's approximation theorem, for each $i=1, 2$, we can approximate
$G_{i, 2k-1}$ on a neighborhood $V_{2k-1}'\subset\subset V_{2k-1}$  of $\overline{\mathbb{D}}_{R_{2k-1}}$
by nonconstant entire function
$\hat{G}_{i, 2k-1}$ within small error $\epsilon_{2k-1}/2023$,
such that 
$\hat{F}_{2k-1}:=(\hat{G}_{1, 2k-1}, \hat{G}_{2, 2k-1})$
is very close to $F_{2k-1}$
\[
||\hat{F}_{2k-1}(z)-F_{2k-1}(z)||_{\mathbb{C}^2}\leqslant 2\epsilon_{2k-1}/2023
\qquad
{\scriptstyle
	(\forall \,z\ \in\ V_{2k-1}')}.
\]

Similarly,
we can find some $z_{2k}\in \mathbb{C}$
with $|z_{2k}|=:R_{2k}>R_{2k-1}+1$ such that
$\pi_2\circ\hat{G}_{2, 2k-1}(z_{2k})
=e_{2, k}$.

\medskip\noindent
{\bf Plan of Step $2k$}.
To modify
$\hat{F}_{2k-1}$ as
${F}_{2k}
:=(\hat{G}_{1, 2k-1}\cdot H_{2k}, \hat{G}_{2, 2k-1})$
by some holomorphic
function $H_{2k}$ defined in a neighborhood $V_{2k}$ of $\overline{\mathbb{D}}_{R_{2k}}$,
such that
${F}_{2k}$ is very close to $\hat{F}_{2k-1}$ and ${F}_{2k-1}$ on $V_{2k-1}'$
\begin{equation}
\label{small error 2023, (2)}
||{F}_{2k}(z)-\hat{F}_{2k-1}(z)||_{\mathbb{C}^2}
\leqslant \frac{\epsilon_{2k-1}}{2023},\quad
||{F}_{2k}(z)-{F}_{2k-1}(z)||_{\mathbb{C}^2}
\leqslant \frac{3\,\epsilon_{2k-1}}{2023}
\qquad
{\scriptstyle
	(\forall \,z\ \in\ V_{2k-1}')},
\end{equation}
whereas
the induced currents $\Phi_{\pi\circ {F}_{2k}\restriction_{\mathbb{D}_{R_{2k}}}}$ 
and 
$\frac{[\pi\circ {F}_{2k}(\mathbb{D}_{R_{2k}})]}{\text{Area}(\pi\circ {F}_{2k}(\mathbb{D}_{R_{2k}}))}$
 both
have mass concentrated near $E_1\times \{e_{2, k}\}$ in a quantitative sense to be specified later.

\smallskip
The method of {\em Step $2k-1$}
works as well in {\em Step $2k$}. 
Fix a neighborhood $U_{2k}$ of $z_{2k}$ such that
\[
\pi\circ\hat{G}_{2, 2k-1}(U_{2k})\,
\subset\,
\mathbb{D}(e_{2, k}, 2^{-2k-2023}):=
\pi_2
\big(
\mathbb{D}(\hat{G}_{2, 2k-1}(z_{2k}), 2^{-2k-2023})
\big). 
\]
Choose some
holomorphic function
$\mathcal{X}_{2k}$ defined on a neighborhood $V_{2k}$ of 
$\overline{\mathbb{D}}_{R_{2k}}$
such that
$|\mathcal{X}_{2k}|<1$ on $\overline{\mathbb{D}}_{R_{2k}}\setminus U_{2k}$ whereas
$|\mathcal{X}_{2k}(z_{2k})|>1$. Take 
$H_{2k}:=1+\mathcal{X}_{2k}^{M_{2k}}$ for some very large integer $M_{2k}\gg 1$
to be specialized later. 
Thus
$H_{2k}$ is almost $1$
on $\overline{\mathbb{D}}_{R_{2k}}\setminus U_{2k}$ whereas
$H_{2k}$ oscillates violently
inside $U_{2k}$ so that the goal of {\em Step} $2k$ can be reached.

\medskip\noindent
{\bf Technical Details in the Proof of Theorem~B.} 
Subsequently, in each Step $\ell\geqslant 1$, 
we find out the aforementioned
$z_\ell$, $R_{\ell}=|z_\ell|>R_{\ell-1}+1$,
$U_\ell$ and an  auxiliary holomorphic function
 $\mathcal{X}_{\ell}$ defined on
$V_\ell\supset \overline{\mathbb{D}}_{R_\ell}$. 
Now we  choose some sufficiently large $M_{\ell}\gg 1$ and determine a very small error bound $0<\epsilon_\ell\ll 1$  such that all the following sophisticated considerations hold true simultaneously.
\begin{itemize}
    \item[($\heartsuit$).] The holomorphic disc 
$\overline{F}_{\ell}=\pi\circ {F}_{\ell}\, :\, {\mathbb{D}_{R_{\ell}}}\longrightarrow X$ satisfies
length-area estimates
\begin{equation}
    \label{key estimate 3}
	\dfrac{L_{\overline{F}_{\ell}}(\omega)}{T_{\overline{F}_{\ell}}(\omega)}
	\leqslant
	2^{-\ell},
	\qquad
	\dfrac{\text{Length}_{\omega}\big(\overline{F}_{\ell}(\partial \mathbb{D}_{R_{\ell}})\big)}{\text{Area}_{\omega}\big(\overline{F}_{\ell}( \mathbb{D}_{R_{\ell}})\big)}
	\leqslant
	2^{-\ell}.
\end{equation}

	\smallskip
	\item[($\diamondsuit$).] 
	Most area of
	the holomorphic disc
    $\overline{F}_{\ell}$
    is, either concentrated near the curve
     $\{e_{1, k}\}\times E_2$ when $\ell=2k-1$ is odd, or  concentrated near
     the curve $E_1\times \{e_{2, k}\}$ when $\ell=2k$ is even.
     Precisely,
     \begin{equation}
     \label{area concentration}
     \dfrac{T_{\overline{F}_{\ell}}
     (\mathbf{1}_{X\setminus  \hat{U}_{\ell}}\cdot \omega)}{T_{\overline{F}_{\ell}}
     (\omega)}
    <
     2^{-\ell},
     \qquad
     	\frac{[\overline{F}_{\ell}(\mathbb{D}_{R_\ell})]}{\text{Area}_{\omega}\big(\overline{F}_{\ell}( \mathbb{D}_{R_\ell})\big)}
     	(\mathbf{1}_{X\setminus  \hat{U}_{\ell}}\cdot \omega)
     	=
     	\frac{\int_{\mathbb{D}_{R_\ell}\setminus \overline{F}_{\ell}^{-1}(\hat{U}_\ell) }\,\overline{F}_{\ell}^*\,\omega}{\int_{\mathbb{D}_{R_\ell}}\,\overline{F}_{\ell}^*\,\omega}
     	<
     2^{-\ell},
     \end{equation}
    where
    $\hat{U}_{\ell}$ is either
    $\pi_1 \big(\mathbb{D}({z_{\ell}}, 2^{-\ell})\big)\times E_2$ (when $\ell=2k-1$)
    or
    $\hat{U}_{\ell}=
    E_1\times
    \pi_2 \big(\mathbb{D}({z_{\ell}}, 2^{-\ell})\big)$ (when $\ell=2k$), and where
    $\mathbf{1}_{X\setminus  \hat{U}_{\ell}}$
    is the characteristic function
    of the set $X\setminus  \hat{U}_{\ell}$.

  	\smallskip
	\item[($\clubsuit$).] 
	By our construction,  $F_{\ell}$ is well-defined on all three domains
	$\mathbb{D}_{R_\ell}\subset\subset V_{\ell}'\subset\subset V_\ell$. We choose $0<\epsilon_\ell\ll 1$   
	shrinking to zero very fast (set $\epsilon_0:=1$) such that 
	\begin{equation}
	\label{epsilons shrinking fast}
	\epsilon_{\ell}<\epsilon_{\ell-1}/2
	\end{equation}
	and such that
	for any small holomorphic perturbation $F_{\ell, \mathsf{p}}$  of $F_{\ell}$ within difference $\epsilon_\ell$ on $V_\ell'$ (i.e.,
	$|F_{\ell, \mathsf{p}}-F_{\ell}|<\epsilon_\ell$ on $V_\ell'$),
	the estimates~\eqref{key estimate 3} and~\eqref{area concentration}
	are stable in the sense:
	\begin{equation}
	    \label{stability condition}
	    \dfrac{L_{\overline{F}_{\ell, \mathsf{p}}}(\omega)}{T_{\overline{F}_{\ell, \mathsf{p}}}(\omega)}
	\leqslant
	2^{-\ell+1},
	\quad
	\dfrac{T_{\overline{F}_{\ell, \mathsf{p}}}
     (\mathbf{1}_{X\setminus  2\hat{U}_{\ell}}\cdot \omega)}{T_{\overline{F}_{\ell, \mathsf{p}}}
     (\omega)}
    <
     2^{-\ell+1},
	\end{equation}
	\begin{equation}
	    \label{stability condition 2} 
	   	\dfrac{\text{Length}_{\omega}\big(\overline{F}_{\ell, \mathsf{p}}(\partial \mathbb{D}_{R_{\ell}})\big)}{\text{Area}_{\omega}\big(\overline{F}_{\ell, \mathsf{p}}( \mathbb{D}_{R_{\ell}})\big)}
	\leqslant
	2^{-\ell+1},
	\qquad
		\frac{\int_{\mathbb{D}_{R_\ell}\setminus \overline{F}_{\ell, \mathsf{p}}^{-1}(2\hat{U}_\ell) }\,\overline{F}_{\ell, \mathsf{p}}^*\,\omega}{\int_{\mathbb{D}_{R_\ell}}\,\overline{F}_{\ell, \mathsf{p}}^*\,\omega}
     	<
     2^{-\ell+1},
	\end{equation}
	where 
$\overline{F}_{\ell, \mathsf{p}}:  {\mathbb{D}_{R_{\ell}}} \rightarrow X$ is the restriction of $\pi\circ {F}_{\ell, \mathsf{p}}$ on $ {\mathbb{D}_{R_{\ell}}}$,
and where $2\hat{U}_{\ell}\supset \hat{U}_{\ell}$ 
 is either
    $\pi_1 \big(\mathbb{D}({z_{\ell}}, 2^{-\ell+1})\big)\times E_2$ (when $\ell=2k-1$)
    or
    $E_1\times
    \pi_2 \big(\mathbb{D}({z_{\ell}}, 2^{-\ell+1})\big)$ (when $\ell=2k$).
\end{itemize}

\medskip

 The requirement~($\heartsuit$) is guaranteed by the key estimates~\eqref{key estimate}, \eqref{length-area estimate 2} and by specializing a parameter $M_\ell\gg 1$ sufficiently large.
 
 The demand~($\diamondsuit$)
 follows from~\eqref{the tricky small disck},~\eqref{estimate T} and~\eqref{bound holo disc are on large R} by specializing  $M_\ell\gg 1$,
 since both 
$ {T_{\overline{F}_{\ell}}
     (\mathbf{1}_{X\setminus \hat{U}_{\ell}}\cdot \omega)}
     $ and $\int_{\mathbb{D}_{R_\ell}\setminus \overline{F}_{\ell}^{-1}(\hat{U}_\ell) }\,\overline{F}_{\ell}^*\,\omega$
are  uniformly bounded from above by some constant $O^+(1)$ independent of the parameter $M_\ell\in \mathbb{Z}_+$.  
Indeed, by our construction, for every $M_\ell\geqslant 1$, we  have
$U_\ell\subset\subset \overline{F}_{\ell}^{-1}(\hat{U}_\ell)$, while
 the sequence of holomorphic maps
$\overline{F}_{\ell}$,
parameterized by $M_{\ell}\in \mathbb{Z}_+$,
converges uniformly to
$\pi\circ\hat{F}_{\ell-1}$
on
$\overline{\mathbb{D}}_{R_\ell}\setminus U_\ell$
because 
$|\mathcal{X}_{\ell}|<1$ on $\overline{\mathbb{D}}_{R_\ell}\setminus U_\ell$.

The existence of $\epsilon_\ell$ in the
  consideration
($\clubsuit$)
can be proved by a {\em reductio ad absurdum} argument,  since the $\mathcal{C}^0$--convergence of a sequence of holomorphic functions
on a larger open set $V_{\ell}'\supset \overline{\mathbb{D}}_{R_\ell} $
can guarantee the $\mathcal{C}^1$--convergence on the smaller compact set $\overline{\mathbb{D}}_{R_\ell}$ by complex analysis.

The estimates~\eqref{small error 2023} and~\eqref{small error 2023, (2)} guarantee that the holomorphic discs $\{F_\ell\}_{\ell\geqslant 1}$  in~\eqref{holo discs are convergent} can
converge to an entire curve
$F=(G_1, G_2): \mathbb{C}\rightarrow \mathbb{C}\times \mathbb{C}$.

For every $\ell\geqslant 1$,
 on $V_\ell'$  we have
\begin{align*}
||F-F_\ell||
&
\leqslant
\sum_{j=\ell}^{\infty}\,
||F_{j+1}-F_j||
\\
\text{[use~\eqref{small error 2023},~\eqref{small error 2023, (2)}]}
\qquad
&
\leqslant
\sum_{j=\ell}^{\infty}\,
3\cdot\epsilon_{j}/2023
\\
\text{[check~\eqref{epsilons shrinking fast}]}
\qquad
&
\leqslant
\sum_{j=\ell}^{\infty}\,
3/2023\cdot
\Big(\frac{1}{2}
\Big)^{j-\ell}\cdot\epsilon_{\ell}
\\
&
<
\epsilon_{\ell}
.
\end{align*}

Thus~\eqref{stability condition} guarantees that
$g=\pi\circ F$ produces a
holomorphic disc
$g_\ell:=g\restriction_{{\mathbb{D}}_{R_\ell} }$
satisfying
\begin{equation}
\label{last estimate}
 \dfrac{L_{g_\ell}(\omega)}{T_{g_\ell}(\omega)}
	\leqslant
	2^{-\ell+1},
	\quad
	\dfrac{T_{g_\ell}
     (\mathbf{1}_{X\setminus  2\hat{U}_{\ell}}\cdot \omega)}{T_{g_\ell}
     (\omega)}
     <
     2^{-\ell+1}.
\end{equation}
Thus the sequence of concentric holomorphic discs $\{g\restriction_{{\mathbb{D}}_{R_\ell} }\}_{\ell\geqslant 1}$ satisfy the length-area conditions~\eqref{length-area condition}, \eqref{length-area-condition-Ahlfors} for producing Nevanlinna/Ahlfors currents.

For any $e_1\in E_1$,
since $\{e_{1, \ell}\}_{\ell\geqslant 1}$
is dense in $E_1$, we can 
find some index
 subsequence 
$\{j_{\ell}\}_{\ell\geqslant 1}\nearrow \infty$
 such that the corresponding
point sequence
$\{e_{1, j_\ell}\}_{\ell\geqslant 1}$
converges to $e_1$.
Note that the sequence of holomorphic discs 
$\{g\restriction{\mathbb{D}_{ R_{2j_\ell-1}}\}_{\ell\geqslant 1}}$ satisfies the the length-area condition~\eqref{length-area condition} by  the first inequality of~\eqref{last estimate}. After passing to some subsequence, we thus obtain some positive closed Nevanlinna current $\mathsf{T}$.
By the second inequality of~\eqref{last estimate},
noting that $2\hat{U}_{2j_\ell-1}$
converges to the curve $e_1\times E_2$ as $\ell\nearrow \infty$,
 $\mathsf{T}$ must have zero mass outside $e_1\times E_2$, i.e., 
 $\mathsf{T}$ is supported on $e_1\times E_2$. By Siu's decomposition theorem,
 $\mathsf{T}$ is proportional to
 $[\{e_1\}\times E_2]$.
 
 Similarly, for every $e_2\in E_2$, we can also select some sequence of increasing radii
 to generate a Nevanlinna current proportional to 
 $[E_1\times \{e_2\}]$.
 
Likewise, we can prove the analogous result about Ahlfors currents by  using~\eqref{stability condition 2} instead of~\eqref{stability condition}.
\qed

\medskip\noindent
{\bf Proof of the Key Lemma.} 
Fix a small disc $\overline{\mathbb{D}}(z_0, 2\delta)$ in $U$. 
Draw a simple closed curve  passing through $z_0$
which bounds a simply connected  domain $\mathsf{D}$ 
 strictly larger than $\overline{\mathbb{D}}_R$ except at $z_0$.
Using the Riemann mapping theorem,
we receive a biholomorphic map
$\psi$ from $\mathsf{D}$ to the unit disc $\mathbb{D}_1$.
By compactness,
the maximum modulus $\mathsf{m}$ of
$|\psi|$ on $\overline{\mathbb{D}}_R\setminus \mathbb{D}(z_0, \delta)$ is strictly less than $1$.
Now we claim that 
\[
\mathcal{X}(
\bullet):=(1+\delta_1)
\cdot \psi\big((1-\delta_2)\cdot \bullet\big)
\]
satisfies our requirement for some  small $\delta_1, \delta_2>0$.

Indeed, we first require that
$(1+\delta_1)\cdot \mathsf{m}<1$.
Next, we choose $\delta_2>0$ sufficiently small such that
\[
\frac{1}{1-\delta_2}\cdot\mathbb{D}(z_0, \delta):=\mathbb{D}
\Big(
\frac{z_0}{1-\delta_2}, \frac{\delta}{1-\delta_2}
\Big)\,
\subset\,
\mathbb{D}(z_0, 2\delta)
\qquad
\Big(\text{equivalently,}\,
\delta_2
<
\frac{\delta}{|z_0|+2\delta}
\Big),
\]
and that
$|\mathcal{X}(
z_0)|=(1+\delta_1)
\cdot |\psi\big((1-\delta_2)\cdot z_0\big)|>1$.
Then $\mathcal{X}$ is well defined on $$\frac{1}{1-\delta_2}\cdot \mathsf{D}:=
\Big\{
\frac{z}{1-\delta_2}\,:\,
z\in \mathsf{D}
\Big\}
\,\,\supset\,\, \overline{\mathbb{D}}_R$$ and satisfies our requirements.
\qed

\setlength\parindent{0em}

\medskip

{\scriptsize

\smallskip
\noindent
S.-Y. Xie

\smallskip
\noindent
Academy of Mathematics and System Science \& Hua Loo-Keng Key Laboratory
		of Mathematics, Chinese Academy of Sciences, Beijing 100190, China; 
		School of Mathematical Sciences, University of Chinese Academy of Sciences, Beijing
100049, China

}

\smallskip
\noindent
\raggedleft{\scriptsize  xiesongyan@amss.ac.cn}


\bigskip

\begin{thebibliography}{XL}{\scriptsize



 
{\bf\bibitem{Alarcon-Forstneric}
{\rm Alarc\'on}}, A.; {\rm Forstneri\v{c}}, F.: 
{\em Oka-1 manifolds.} 
arXiv:2303.15855, 2023.


\medskip

{\bf\bibitem{Berczi-Kirwan}
{\rm Bérczi}}, G.; {\rm Kirwan}, F.: 
{\em Non-reductive geometric invariant theory and hyperbolicity.} 
arXiv:1909.11417, 2019.


\medskip



{\bf\bibitem{Birkhoff}
{\rm Birkhoff}}, G. D.: 
{\em D\'emonstration d’un th\'eor\`eme \'el\'ementaire sur les fonctions enti\`eres.}  
C.R. Acad.
Sci. Paris, 189:473--475, 1929.


\medskip

{\bf\bibitem{Bloch-1926}
{\rm Bloch}}, A.:
{\em Sur les syst\`emes de fonctions uniformes satisfaisant \`a l’\'equation d’une vari\'t\'e alg\'ebrique dont l’irr\'egularit\'e d\'epasse la dimension.}  J. de
Math., 5, 19--66, 1926.


\medskip


{\bf\bibitem{Brody-1978}
{\rm Brody}}, R.:
{\em Compact manifolds and hyperbolicity.}
 Trans. Amer. Math. Soc. 235, 213--219,  1978.
 

\medskip


{\bf\bibitem{Brotbek-2016}
{\rm Brotbek}}, D.:
{\em Explicit symmetric differential forms on complete intersection varieties and applications.}  Math. Ann. 366(1), 417--446, 2016.

\medskip

{\bf\bibitem{Brotbek-2017}
{\rm Brotbek}}, D.:
{\em On the hyperbolicity of general hypersurfaces.}  Publ. Math. Inst. Hautes Études Sci. 126, 1--34,
 2017.

\medskip




{\bf\bibitem{Brotbek-Lionel}
{\rm Brotbek}}, D.; {\rm Darondeau}, L.:
{\em Complete intersection varieties with ample cotangent bundles.} (arXiv:1511.04709) Invent. Math. 212, 913--940, 2018.

\medskip


{\bf\bibitem{Brunella-1999}
{\rm Brunella}}, M.:
{\em Courbes entieres et feuilletages holomorphes.}
L’Enseignement Mathematique, 45:195--216, 1999.

\medskip

{\bf\bibitem{CHX-2023}
	{\rm Chen}}, Z.; {\rm Huynh}, D. T.; {\rm Xie}, S.-Y.:
{\em Universal Entire Curves in Projective Spaces with Slow Growth.}  Journal of Geometric Analysis,  33, 308, 2023.

\medskip


{\bf\bibitem{Debarre-2005}
	{\rm Debarre}}, O.: 
{\em Varieties with ample cotangent bundle.}   Compos. Math. 141(6), 1445--1459, 2005.

{\bf\bibitem{Demailly-2011}
	{\rm Demailly}}, J.-P.:
{\em Holomorphic Morse inequalities and the Green-Griffiths-Lang conjecture.}  Pure Appl. Math. Q. 7, 1165--1207; Special Issue: In memory of Eckart Vihweg,
2011.

\medskip

{\bf\bibitem{DMR}
{\rm Diverio}}, S.; {\rm Merker}, J.; {\rm Rousseau}, E.;
{\em Effective algebraic degeneracy.} 
Invent. Math.,
180, 161--223, 2010.

\medskip

{\bf\bibitem{Dinh-Sibony-2018}
{\rm Dinh}}, T-C.; {\rm Sibony}, N.;
{\em Unique ergodicity for foliations in $\mathbb{P}^2$ with an invariant curve.} Invent. Math.,
211(1):1--38, 2018.

\medskip

  

{\bf\bibitem{Dinh-Sibony-2020}
{\rm Dinh}}, T-C.; {\rm Sibony}, N.;
{\em Some open problems on holomorphic foliation theory.}
Acta Math. Vietnam., 45(1):103--112, 2020.


\medskip

{\bf\bibitem{Dujardin-ICM}
{\rm Dujardin}}, R.:
{\em Geometric methods in holomorphic dynamics.} 
Proceedings of the ICM 2022.
 


\medskip


{\bf\bibitem{Duval-degree6}
{\rm Duval}}, J.:
{\em Une sextique hyperbolique dans $\mathbb{P}^3(\mathbb{C})$.}
  Math. Ann. 330.3, 473--476, 2004. 
 

\medskip


{\bf\bibitem{Duval-2006}
{\rm Duval}}, J.:
{\em Singularités des courants d'Ahlfors.}  Ann. Sci. Éc. Norm. Supér. (4) 39 (3), 527--533, 2006


\medskip


{\bf\bibitem{Duval-2008}
{\rm Duval}}, J.:
{\em Sur le lemme de Brody.}  Invent. Math., 173(2):305--314, 2008.


\medskip

{\bf\bibitem{Duval-2017}
{\rm Duval}}, J.:
{\em Around brody lemma.} arXiv:1703.01850, 2017.

\medskip


{\bf\bibitem{DH-2018}
{\rm Duval}}, J.; {\rm Huynh}, D. T.:
{\em A geometric second main theorem.}  Math. Ann., 370(3-4):1799–1804, 2018.



\medskip


{\bf\bibitem{Oka-book}
	{\rm Forstneri\v{c}}}, Franc:
{\em Stein manifolds and holomorphic mappings.}
Springer, Cham; Ergebnisse der Mathematik und ihrer Grenzgebiete Vol. 56 (2nd edition). xiv+562 pp. 2017.

\medskip





{\bf\bibitem{Gaier}
	{\rm Gaier}}, D.:
{\em Lectures on Complex Approximation.}   Birkhäuser, Boston, 1987.

\medskip

{\bf\bibitem{Green-Griffiths}
	{\rm Green}}, M.; {\rm Griffiths}, P.:
{\em Two applications of algebraic geometry to entire holomorphic mappings.}  In The Chern
Symposium 1979 (Proc. Internat. Sympos., Berkeley, Calif., 1979), pages 41--74. Springer, New York-Berlin, 1980.

\medskip


{\bf\bibitem{Guo-Xie-1}
	{\rm Guo}}, B.; {\rm Xie}, S.-Y.:
{\em Universal holomorphic maps with slow growth I. An Algorithm.}
arXiv:2306.11193 (to appear in Math. Ann.), 2023. 

\medskip





{\bf\bibitem{Hormander}
	{\rm H\"ormander}}, L.:
{\em The fully nonlinear Cauchy problem with small data.}  Bol. Soc. Brasil. Mat. (N.S.),
20(1):1--27, 1989.




\medskip


{\bf\bibitem{Tuan-IMRN}
	{\rm Huynh}}, D. T.:
{\em Examples of Hyperbolic Hypersurfaces of Low Degree in Projective Spaces.} 
International Mathematics Research Notices, Volume 2016, Issue 18, pp. 5518--5558, 2016.


\medskip




{\bf\bibitem{Huynh-PhD}
	{\rm Huynh}}, D. T.:
{\em Sur le Second Théorème Principal.}  Ph.D. Thesis, Orsay, 2016.

\medskip


{\bf\bibitem{HV-2021}
	{\rm Huynh}}, D. T.; {\rm Vu}, D.-V.:
{\em On the set of divisors with zero geometric defect.} Journal für die reine und angewandte Mathematik, vol. 2021, no. 771,  pp. 193--213, 2021.

\medskip



{\bf\bibitem{HX-2021}
	{\rm Huynh}}, D. T.;  {\rm Xie}, S.-Y.:
{\em On Ahlfors currents.} 
Journal de Mathématiques Pures et Appliquées,
Volume 156,
Pages 307--327,
2021.

\medskip

{\bf\bibitem{Kleiner}
{\rm Kleiner}}, B.:
{\em Hyperbolicity using minimal surfaces.}  Preprint.

\medskip

{\bf\bibitem{Kobayashi-1998}
{\rm Kobayashi}}, S.:
{\em Hyperbolic Complex Spaces.}  Grundlehren der Mathematischen Wissenschaften, Vol. 318, Springer-Verlag, Berlin, 1998. 

\medskip

{\bf\bibitem{Kusakabe-2017}
{\rm Kusakabe}}, Y.:
{\em Dense holomorphic curves in spaces of holomorphic maps and applications to universal
maps.}
  Internat. J. Math., 28(4):1750028, 15, 2017.



\medskip

{\bf\bibitem{Lu-Yau}
{\rm Lu}}, S.; {\rm Yau}, S.-T.:
{\em Holomorphic curves in surfaces of general type.}   Proc.
Nat. Acad. Sci. USA, 87, 80--82, 1990.

 
\medskip

{\bf\bibitem{McQuillan-1998}
{\rm McQuillan}}, M.:
{\em Diophantine approximations and foliations.}  Inst. Hautes Etudes Sci. Publ. Math., (87):121--174,
1998.


\medskip

{\bf\bibitem{McQuillan-ICM}
{\rm McQuillan}}, M.:
{\em Integrating $\partial \bar{\partial}$.}  In Proceedings of the International Congress of Mathematicians, Vol. I (Beijing, 2002), Higher Ed. Press, pp. 547--554, 2002.


\medskip

{\bf\bibitem{Merker-2009}
{\rm Merker}}, J.:
{\em Low pole order frames on vertical jets of the universal hypersurface.} Annales de l’Institut Fourier 59(3), 1077--1104, 2009.


\medskip

{\bf\bibitem{Merker-2015}
{\rm Merker}}, J.:
{\em Algebraic differential equations for entire holomorphic curves in projective hyper-surfaces of general type: optimal lower degree bound.} (arXiv:1005.0405) In Geometry and analysis on manifolds, Progress in Mathematics, vol. 308, 41--142, Birkhäuser/Springer, Cham, 2015.



\medskip

{\bf\bibitem{Riedl-Yang}
{\rm Riedl}}, E.; {\rm Yang}, D.: 
{\em Applications of a Grassmannian technique to hyperbolicity, Chow equivalency, and Seshadri constants.} (arXiv:1806.02364) 
  Journal of Algebraic Geometry 31, 1--12,  2022.

\medskip

{\bf\bibitem{Ru-2021}
{\rm Ru}}, M.:
{\em Nevanlinna theory and its relation to Diophantine approximation.}
Second edition. World Scientific Publishing Co. Pte. Ltd., Hackensack, NJ, xvi+426 pp, 2021.



\medskip

{\bf\bibitem{Sibony-email}
{\rm Sibony}}, N.:
 Private communication, January 2021.

\medskip


{\bf\bibitem{Siu-1974}
{\rm Siu}}, Y.-T.:
{\em Analyticity of sets associated to Lelong numbers and the extension of closed positive currents.} Invent.
Math., 27:53--156, 1974.

\medskip




{\bf\bibitem{Siu-1995}
{\rm Siu}}, Y.-T.:
{\em Hyperbolicity problems in function theory.} In: Chan, K.-Y., Liu, M.-C. (eds.) Five Decades as a Mathematician and Educator---On the 80th Birthday of Professor Yung-Chow Wong, pp. 409--514. World Scientific, Singapore, 1995.

\medskip

{\bf\bibitem{Siu-ICM2002}
{\rm Siu}}, Y.-T.:
{\em Some Recent Transcendental Techniques in Algebraic and Complex
Geometry.} In Proceedings of the International Congress of Mathematicians, Beijing, China,
August 20-28, 2002, Volume I: 439--448.

\medskip

{\bf\bibitem{Siu-Abel}
{\rm Siu}}, Y.-T.:
{\em Hyperbolicity in Complex Geometry.}   In: Laudal, O.A., Piene, R. (eds) The Legacy of Niels Henrik Abel. Springer, Berlin, Heidelberg, 2004.


\medskip

{\bf\bibitem{Siu-2015}
{\rm Siu}}, Y.-T.:
{\em Hyperbolicity of generic high-degree hypersurfaces in complex projective spaces.}  Invent. Math. 202, no.~3, 1069--1166, 2015.

\medskip

{\bf\bibitem{Vojta-1999}
{\rm Vojta}}, P.:
{\em  Nevanlinna theory and Diophantine approximation.}  In: Siu,
Y.-T., Schneider, M. (ed) Several complex variables (Berkeley, CA, 1995–
1996), 535–564, Math. Sci. Res. Inst. Publ., 37, Cambridge Univ. Press,
Cambridge, 1999.

\medskip

{\bf\bibitem{Xie-2018}
{\rm Xie}}, S.-Y.:
{\em On the ampleness of the cotangent bundles of complete intersections.}
(arXiv:1510.06323)
 Invent. Math. 212, 941--996, 2018.

\medskip

{\bf\bibitem{Zaidenberg-1989}
{\rm Zaidenberg}}, M.:
{\em Stability of hyperbolic embeddedness and construction
of examples.}  Math. USSR Sbornik, 63, 351--361, 1989.

}



\end{thebibliography}
\end{document}